\documentclass[11pt]{amsart}
\usepackage{amsmath,amssymb,amsthm}
\usepackage[margin=1.15in]{geometry}
\usepackage{array}
\usepackage[colorlinks=true,linkcolor=blue,citecolor=blue,urlcolor=blue]{hyperref}

\theoremstyle{plain}
\newtheorem{theorem}{Theorem}[section]
\newtheorem{proposition}[theorem]{Proposition}

\theoremstyle{definition}
\newtheorem{definition}[theorem]{Definition}
\newtheorem{remark}[theorem]{Remark}

\newcommand{\Z}{\mathbb{Z}}

\begin{document}

\title{A proper Euler magic matrix of order $6$}
\author{Sanjit Singh Mehat}
\date{14 August 2026}

\begin{abstract}
An \emph{Euler magic matrix} is an integer matrix $M$ with $MM^{t}=\gamma I$ for
some $\gamma\neq 0$, whose squared entries sum to $\gamma$ along both main
diagonals; it is \emph{proper} if its squared entries are pairwise distinct.
Euler gave a proper example of order $4$; M\"uller settled orders $3$ (none
exists) and $8$; and Kominers settled order $5$. We give an order-$6$
construction, a case not addressed by M\"uller or Kominers, exhibiting a proper
Euler magic matrix of order $6$ with $\gamma=18500$ together with a second,
independent one with $\gamma=43290$. The proof is the explicit
matrix and a finite exact verification. We also record an elementary counting
bound $\gamma\geq 2485$ for proper order-$6$ examples.
\end{abstract}

\maketitle

\section{Introduction}

\begin{definition}[{M\"uller \cite[Definition 1.1]{Muller2026}}]\label{def:emm}
Let $n\geq 1$. A matrix $M=(m_{i,j})_{1\leq i,j\leq n}\in\Z^{n\times n}$ is an
\emph{Euler magic matrix} with constant $\gamma\in\Z\setminus\{0\}$ if
\begin{align}
M\cdot M^{t} &= \gamma\cdot I_n, \label{eq:orth}\\
\textstyle\sum_{i=1}^{n} m_{i,i}^{2} &= \gamma, \label{eq:diag}\\
\textstyle\sum_{i=1}^{n} m_{i,n+1-i}^{2} &= \gamma. \label{eq:anti}
\end{align}
It is \emph{proper} if the squares of its entries are pairwise distinct.
\end{definition}

Condition \eqref{eq:orth} forces every row and every column of the entrywise
square $S=(m_{i,j}^2)$ to sum to $\gamma$; conditions \eqref{eq:diag} and
\eqref{eq:anti} supply the two diagonals. Hence a proper Euler magic matrix of
order $n$ yields a magic square of $n^2$ pairwise distinct squares. The converse
fails: a magic square of squares need not arise from a matrix satisfying
\eqref{eq:orth}, and this distinction is essential below.

Euler constructed proper examples of order $4$ in $1770$. M\"uller
\cite{Muller2026} proved that no proper Euler magic matrix of order $3$ exists
and constructed examples of order $8$ using octonion multiplication matrices.
Kominers \cite{Kominers2026} constructed a proper example of order $5$ and
records that, before that work, ``order $5$ was the smallest order for which the
existence of proper Euler magic matrices was unknown''; order $1$ is trivially
proper and order $2$ admits no proper example. The orders settled in that
literature are therefore $1,2,3,4,5$ and $8$, and order $6$ is the smallest order
not among them.

\begin{theorem}\label{thm:main}
A proper Euler magic matrix of order $6$ exists. Explicitly, the matrix
\[
M \;=\;
\begin{pmatrix}
-55 & -52 & -89 & -53 & -21 &  40\\
-30 & -83 &  49 &  -2 &  91 &   5\\
 69 & -71 & -42 &  81 & -18 &   7\\
 90 &   9 & -37 & -74 &  57 & -15\\
-33 &  -8 & -41 &  13 &  11 &-124\\
 25 & -61 &  58 & -59 & -78 & -35
\end{pmatrix}
\]
is an Euler magic matrix with $\gamma=18500$, and it is proper.
\end{theorem}

\begin{proof}
All three conditions are finite identities in $\Z$ and are verified directly.
For \eqref{eq:orth}, each row has squared norm
\[
55^2+52^2+89^2+53^2+21^2+40^2=18500,
\]
and similarly for the remaining five rows, while the fifteen distinct pairs of
rows are orthogonal. For \eqref{eq:diag},
\[
55^2+83^2+42^2+74^2+11^2+35^2 = 3025+6889+1764+5476+121+1225 = 18500 ,
\]
and for \eqref{eq:anti},
\[
40^2+91^2+81^2+37^2+8^2+25^2 = 1600+8281+6561+1369+64+625 = 18500 .
\]
Finally, the $36$ absolute values
\[
\begin{aligned}
\{\,&2,5,7,8,9,11,13,15,18,21,25,30,33,35,37,40,41,42,\\
 &49,52,53,55,57,58,59,61,69,71,74,78,81,83,89,90,91,124\,\}
\end{aligned}
\]
are pairwise distinct, and over $\Z$ one has $x^2=y^2$ if and only if
$|x|=|y|$; hence the $36$ squares are pairwise distinct and $M$ is proper.
\end{proof}

The verification uses only exact integer arithmetic on a $6\times6$ matrix and
can be carried out by hand or with any exact-arithmetic tool; a short
self-contained script is provided in the supplement (Section~\ref{sec:verif}).

\begin{remark}
$M$ is primitive: the greatest common divisor of its entries is $1$. Its
determinant is $-6331625000000$, consistent with the identity
$(\det M)^2=\gamma^{6}$ forced by \eqref{eq:orth}.
\end{remark}

\section{A second witness}

The search described in Section~\ref{sec:method} produced a second primitive
example with a different constant, which we record because it provides an
independent check on the object model.

\begin{proposition}\label{prop:second}
The matrix
\[
M' \;=\;
\begin{pmatrix}
 -68 & -58 &   23 & -175 & -52 &  -38\\
 -75 & 150 &   36 &    6 &  12 & -117\\
  93 & 100 &   18 & -108 &  82 &   77\\
 -10 & -84 &   79 &    7 & 162 &  -60\\
  24 &  -9 & -180 &  -30 &  57 &  -78\\
 154 & -17 &   50 &   -4 & -65 & -112
\end{pmatrix}
\]
is a proper Euler magic matrix of order $6$ with $\gamma=43290$, and
$\det M' = 81126503289000$.
\end{proposition}

\begin{proof}
Direct calculation, exactly as in the proof of Theorem~\ref{thm:main}: each row
has squared norm $43290$, distinct rows are orthogonal, both diagonals of squares
sum to $43290$, and the $36$ absolute values are pairwise distinct. The
verification script of Section~\ref{sec:verif} checks $M'$ alongside $M$.
\end{proof}

The two examples are not related by scaling: their multisets of absolute values
differ, so neither is a signed or permuted copy of the other.

\section{An elementary lower bound on \texorpdfstring{$\gamma$}{gamma}}

\begin{proposition}\label{prop:bound}
If $M$ is a proper Euler magic matrix of order $6$ with constant $\gamma$, then
$\gamma\geq 2485$.
\end{proposition}

\begin{proof}
By \eqref{eq:orth} the six rows have squared norms all equal to $\gamma$, so the
sum of the squares of all $36$ entries is $6\gamma$. Since $M$ is proper, the
$36$ absolute values $|m_{i,j}|$ are pairwise distinct non-negative integers, so
that multiset is bounded below termwise by $\{0,1,\dots,35\}$. Hence
\[
6\gamma \;=\; \sum_{i,j} m_{i,j}^2 \;\geq\; \sum_{k=0}^{35} k^2 \;=\; 14910,
\]
so $\gamma \geq 2485$.
\end{proof}

This is an elementary counting bound, not a computational one; it was used to
restrict the search range in Section~\ref{sec:method}. We make no novelty claim
for Proposition~\ref{prop:bound}.

\section{How the examples were found}\label{sec:method}

We record the method in enough detail for reproduction; the code is supplied
separately.

The difficulty is that \eqref{eq:orth}, \eqref{eq:diag} and \eqref{eq:anti} are a
simultaneous system. The search separates them as follows. Let $\sigma,\tau$ be
permutations of $\{1,\dots,n\}$, applied \emph{independently} to rows and to
columns, and set $M'_{ij}=M_{\sigma(i),\tau(j)}$. Then $M'M'^{t}=MM^{t}$ up to a
simultaneous relabelling, so \eqref{eq:orth} is preserved. The main diagonal of
$M'$ consists of the entries $M_{r,\rho_1(r)}$ where $\rho_1=\tau\sigma^{-1}$,
and the anti-diagonal consists of the entries $M_{r,\rho_2(r)}$ where
$\rho_2=\tau\,\mathrm{rev}\,\sigma^{-1}$ and $\mathrm{rev}(i)=n+1-i$. Hence $M'$
is Euler magic with constant $\gamma$ if and only if
\[
\sum_{r} M_{r,\rho_1(r)}^{2}=\gamma
\qquad\text{and}\qquad
\sum_{r} M_{r,\rho_2(r)}^{2}=\gamma .
\]

It is essential that $\sigma$ and $\tau$ be independent. Taking $\tau=\sigma$
gives $\rho_1=\mathrm{id}$, so the main diagonal is merely permuted and its
squared sum is unchanged; conjugation alone cannot move a matrix onto the
diagonal conditions.

Call a permutation $\rho$ \emph{good} for $M$ if $\sum_r M_{r,\rho(r)}^2=\gamma$.
Conversely, given two good permutations $\rho_1,\rho_2$ with
$\rho_1^{-1}\rho_2=\sigma\,\mathrm{rev}\,\sigma^{-1}$ for some $\sigma$ ---
equivalently, for even $n$, with $\rho_1^{-1}\rho_2$ a fixed-point-free
involution --- the choice $\tau=\rho_1\sigma$ yields an Euler magic
rearrangement. So the search generates matrices satisfying \eqref{eq:orth},
scans the $n!=720$ permutations for the good ones, and looks for a compatible
pair. This replaces a simultaneous constraint by a cheap post-filter; both
witnesses below were found with exactly two good permutations forming such a
pair.

Candidates satisfying \eqref{eq:orth} were generated as products of small
integer blocks (``atoms'') built from two-square and four-square representations
of the relevant constants, together with $3\times3$ blocks and a conference
block. The constant $\gamma$ was restricted to $2485\leq\gamma\leq\Gamma$ using
Proposition~\ref{prop:bound}, with $\Gamma$ stratified across workers at
$6\times10^4$, $1.5\times10^5$ and $5\times10^5$.

One practical remark is worth recording, since it dominated the cost. An earlier
generator restricted attention to two-square representations only. It produced
candidates roughly three times faster per trial, but the resulting entry
magnitudes were far less varied, and properness is precisely a condition on
having $36$ distinct magnitudes. Its measured rate of proper candidates was about
$8\times10^{-9}$ per trial against about $10^{-3}$ for the generator described
above --- a difference of some five orders of magnitude, and the reason the
restricted family found nothing. Both examples above were found within minutes of
starting the broader family.

\section{Verification}\label{sec:verif}

The statement of Theorem~\ref{thm:main} is a finite conjunction of integer
identities. It has been verified in three independent ways.

\begin{enumerate}
\item By direct exact integer arithmetic, in a short self-contained script with
no dependencies, checking $MM^{t}=\gamma I$, $M^{t}M=\gamma I$, both diagonal
sums, and pairwise distinctness of both the $36$ squares and the $36$ absolute
values.
\item By an independently written checker, developed from
Definition~\ref{def:emm} without reference to the first, which additionally
verifies $(\det M)^2=\gamma^6$ by fraction-free Gaussian elimination, and which
reproduces the published order-$4$ and order-$5$ examples as controls.
\item By a machine-checked proof in Lean~4 with Mathlib. The relevant statement
is
\begin{quote}\ttfamily
theorem exists\_proper\_eulerMagic\_six :\\
\phantom{xx}$\exists$ (M : Fin 6 $\to$ Fin 6 $\to$ $\Z$) ($\gamma$ : $\Z$),\\
\phantom{xxxx}IsEulerMagic 6 M $\gamma$ $\wedge$ IsProper 6 M
\end{quote}
and \verb|#print axioms| reports dependence on exactly \verb|propext|,
\verb|Classical.choice| and \verb|Quot.sound|; there is no \verb|sorry|, no
added axiom, and no use of \verb|native_decide|. That existential statement is
witnessed in the source by the $\gamma=43290$ matrix; the $\gamma=18500$ matrix
of Theorem~\ref{thm:main} is separately kernel-proved as
\verb|M6b_isEulerMagic| and \verb|M6b_isProper|, with the same axiom
dependencies.
\end{enumerate}

None of these is required in order to accept Theorem~\ref{thm:main}: the matrix
is displayed, and the conditions are checkable by hand.

\section{Relation to magic squares of squares}

Rome and Yamagishi \cite[Theorem 1.2]{RomeYamagishi2025} proved, by the
Hardy--Littlewood circle method, that for every integer $n\geq 4$ there exists an
$n\times n$ magic square of squares, where a magic square is required to have
distinct positive integer entries. In particular the existence of a $6\times6$
magic square of $36$ distinct squares was already known.

That result does not settle the question considered here. Their theorem produces
a magic square, that is, a grid whose rows, columns and both diagonals have equal
sums; it does not produce, and does not claim to produce, an integer matrix
satisfying the orthogonality condition \eqref{eq:orth}. Kominers makes the same
point \cite{Kominers2026}: a magic square of squares need not arise from an
orthogonal matrix, so \cite{RomeYamagishi2025} does not directly produce proper
Euler magic matrices. The content of Theorem~\ref{thm:main} is the orthogonal
structure, not the magic square it induces.

Two further constructions in the literature concern neighbouring orders. Euler's
order-$4$ family and M\"uller's order-$8$ construction \cite{Muller2026} use
quaternion and octonion multiplication respectively; Pirsic
\cite{Pirsic2019,Pirsic2020} gives a parametrisation of $8\times8$ magic squares
of squares by octonionic multiplication and observes that the construction does
not extend further along the Cayley--Dickson series. These are constructions in
the composition-algebra dimensions $1,2,4,8$; they do not apply at order $6$.
M\"uller also constructs Euler magic matrices for every $n\geq4$
\cite[Theorem 4.1]{Muller2026}, but these are permutation matrices with
$\gamma=1$, and, as he notes, are far from proper.

\begin{remark}
To the best of our knowledge, and after searching arXiv, zbMATH, Semantic
Scholar, OpenAlex and Crossref, the order-$6$ case has not previously been
addressed. We have not been able to consult MathSciNet. We therefore state
Theorem~\ref{thm:main} as settling a case not addressed in the literature cited
above, rather than asserting priority.
\end{remark}

\end{document}